\documentclass[a4paper,10pt]{article}
\usepackage[utf8]{inputenc}
\usepackage[english]{babel}
\usepackage{amsfonts}
\usepackage{amssymb, amsthm}
\usepackage{xcolor}
\usepackage{marginnote}
\usepackage{amsmath}
\usepackage{a4wide}
\usepackage[affil-it]{authblk}

\numberwithin{equation}{section}

\usepackage[colorlinks=true,linktocpage,pdfpagelabels,
bookmarksnumbered,bookmarksopen]{hyperref}
\definecolor{ForestGreen}{rgb}{0.1,0.6,0.05}
\definecolor{EgyptBlue}{rgb}{0.063,0.1,0.6}
\hypersetup{
	colorlinks=true,
	linkcolor=EgyptBlue,         
	citecolor=ForestGreen,
	urlcolor=olive
}

\usepackage[hyperpageref]{backref}

\usepackage{accents}

\addto\captionsenglish{}
\numberwithin{equation}{section}
\newtheorem{thm}{Theorem}
\newtheorem{lemma}{Lemma}
\newtheorem{prop}{Proposition}
\newtheorem{cor}{Corollary}
\theoremstyle{definition}

\newtheorem{remark}[thm]{Remark}

\allowdisplaybreaks

\title{On an inverse  spectral problem and a  generalized  Sturm's nodal theorem for nonlinear boundary value problems\\ \medskip}
\author[1]{Y.Sh.~Ilyasov \thanks{E-mail: \texttt{ilyasov02@gmail.com}}}
\author[2]{N. F.~Valeev \thanks{E-mail: \texttt{valeevnf@mail.ru}}\thanks{N. F.~Valeev  was
partly supported by RFBR grant N18-51-06002 Az-a}}

\affil[1]{\small Institute of Mathematics of UFRC RAS,  450008, Ufa, Russia
\newline
Instituto de Matem\'atica e Estat\'istica.
Universidade Federal de Goi\'as,
74001-970, Goiania, Brazil}

\affil[2]{\small Institute of Mathematics of UFRC RAS,  450008, Ufa, Russia, \newline
Bashkir State University, 450076, Ufa, Russia}
\date{}

\begin{document}
\maketitle

\begin{abstract}
We consider an inverse optimization spectral problem for the Sturm-Liouville operator $\mathcal{L}[q] u:=-u''+q(x)u$   subject to the separated boundary conditions. In the main result, we prove that  this problem is related to the existence of solutions of the boundary value problems for the nonlinear equations of the form
$-u''+q_0(x) u=\lambda u+\sigma u^3$ with $\sigma=1$ or $\sigma=-1$.
The key outcome of this relationship is a generalized Sturm's nodal theorem for the nonlinear boundary value problems.
\end{abstract}

\maketitle

	%
%
%

%
%


\section{Introduction}
In the present paper we are concerned with the relations between the existence of solutions to the so called \textit{inverse optimization spectral problem} (\cite{ilValMatZam, ilVal}) for the Sturm-Liouville operator
\begin{equation} \label{eq:S}
\mathcal{L}[q] u:=-u''+q(x)u\,\,\,\,
\end{equation}
subject to the separated boundary conditions
\begin{eqnarray}
	&&u(0)\cos\alpha+u'(0) \sin \alpha=0,\\
	&& u(\pi)\cos\alpha+u'(\pi)\sin\alpha=0,
	\end{eqnarray}
and the existence of weak solutions of the nonlinear boundary value problem:
\begin{equation} \label{eq:Nonl}\tag{$NP_{\delta}$}
\begin{cases}
-u''+q_0(x) u=\lambda u+\delta u^3,~~~ x\in (0,\pi),	\\
~~u(0)\cos\alpha+u'(0) \sin \alpha=0,\\
~~u(\pi)\cos\alpha+u'(\pi)\sin\alpha=0,
\end{cases}
\end{equation}
with $\delta=1$ or $\delta=-1$.
	
We suppose that $q \in L^2:=L^2(0,\pi)$. Under these conditions   $\mathcal{L}[q]$ defines a self-adjoint operator on the Hilbert space  $L^2(0,\pi)$ (see, e.g., \cite{edmund, Poschel, zetl}), so that its spectrum consists of an infinite sequence of eigenvalues $\sigma_p(\mathcal{L}[q]):=\{\lambda_i(q) \}_{i=1}^{\infty}$ which can be ordered as follows  $\lambda_1(q)<\lambda_2(q)< \ldots$.
Furthermore, to each eigenvalue $\lambda_k(q)$ corresponds a unique (up to a normalization constant) eigenfunction $\phi_k(q)$
 which has exactly $k-1$ zeros in $(0,\pi)$.

The inverse spectral problem which consists in recovering of the potential
 $q(x)$ from a knowledge of the spectral data  $\{\lambda_i(q) \}_{i=1}^{\infty}$
is a classical problem and, beginning with the celebrated  papers by Ambartsumyan \cite{ambar} in 1929, Borg in 1946 \cite{borg}, Gel'fand \& Levitan \cite{gelL} in 1951, it received a lot of attention.

%


It is well known (see, e.g., \cite{borg, gelL}) that the inverse spectral problem with only finite  given set of eigenvalues
$\{\lambda_i \}_{i=1}^{m}$, $m<+\infty$   have infinitely many
solutions  and in general is meaningless.
However, if one assume that a certain information about the potential $q$ is known  in advance, for instance, it is given an approximate function $q_0$ of the potential $q$, then it is natural to consider  the following \textit{m-parameter inverse optimization spectral problem}: for a given $q_0$ and $\{\lambda_i \}_{i=1}^{m}$, $m<+\infty$, find a potential $\hat{q}$ closest to $q_0$ in a prescribed norm, such that $\lambda_i=\lambda_i(\hat{q})$ for all $i=1, \ldots, m$.

 In the present paper, we  study the following 1-parameter variant of this problem:
Let $k\geq 1$.
\medskip

\par\noindent
$P(k)$:\,\,\textit{For a given $\lambda \in \mathbb{R}$ and $q_0 \in L^2(0,\pi)$, find   a potential  $\hat{q} \in L^2(0,\pi)$ such that  $\lambda=\lambda_k(\hat{q})$ and }

\begin{equation*}
	\|q_0-\hat{q} \|_{L^2}=\inf\{\|q_0-q\|_{L^2}:~~ \lambda=\lambda_k(q), ~~q \in L^2(0,\pi)\}.
\end{equation*}

\medskip


 Our main result is as follows

\begin{thm}\label{thm1}
Let $q_0 \in L^2(0,\pi)$ be a given potential, $k\geq 1$ .
Then
\par
$(1^o)$ for any $\lambda \in \mathbb{R}$,
there exists  a solution  $\hat{q}$ of inverse optimization spectral problem $P(k)$.

Furthermore,
\par
$(2^o)$  for $\lambda< \lambda_k(q_0)$,  there exists a non-zero weak  solution $\hat{u}_\delta $  of \eqref{eq:Nonl}$|_{\delta=1}$,
and for $\lambda> \lambda_k(q_0)$, there exists a non-zero weak  solution $\hat{u}_\delta $  of \eqref{eq:Nonl}$|_{\delta=-1}$  so that the following explicit formula holds
$$
\hat{q}=q_0-\delta\hat{u}_\delta^2~~\mbox{a.e. in}~~ (0,\pi).
$$
\par
$(3^o)$ The solution $\hat{u}_\delta(x)$ of \eqref{eq:Nonl}, $\delta=\pm 1$ possess precisely $k-1$ roots in $(0,\pi)$.
\par

\end{thm}
The case $k=1$ with the zero Dirichlet boundary conditions has been  studied in our recent papers \cite{ilValMatZam, ilVal} where the existence and uniqueness of solution $\hat{q}$ has been proven in the case $\lambda> \lambda_1(q_0)$. Furthermore, in this case, stronger result holds, namely: the uniqueness theorem for \eqref{eq:Nonl}$|_{\delta=-1}$  is satisfied so that \eqref{eq:Nonl}$|_{\delta=-1}$  possess a unique positive solution for $\lambda> \lambda_1(q_0)$.
\begin{remark}
\textit{For $\lambda=\lambda_k(q_0)$, problem $P(k)$ becomes trivial since in this case $\hat{q}=q_0$.}
\end{remark}
\begin{remark}
 \textit{Any eigenfunction $\phi_k(q)$ of $\mathcal{L}[q]$, as well as any weak solution $\hat{u} \in W^{1,2}_0(0,\pi)$ of \eqref{eq:Nonl}, $\delta=\pm 1$, obeys  $C^{1, \beta}[0,1]$-regularity, $\beta \in (0,1)$ (see e.g. \cite{GilTrud}).}
\end{remark}
 It is worth pointing out the following result, which in itself is notably important
\begin{lemma} \label{lem1}
Let $k\geq 1$. Then
\begin{description}
\item[(i)]  for any $\lambda\geq \lambda_k(q_0)$, problem \eqref{eq:Nonl}$|_{\delta=1}$ has no non-zero weak solution with $k-1$ or less  roots in $(0,\pi)$; 	
\item[(ii)]  for any $\lambda\leq \lambda_k(q_0)$, problem \eqref{eq:Nonl}$|_{\delta=-1}$ has no non-zero weak solution with $k-1$ or more  roots in $(0,\pi)$ .
	
\end{description}
\end{lemma}
 From Theorem \ref{thm1} and Lemma \ref{lem1} it follows
\begin{cor} \label{cor1} Assume $q_0 \in L^2(0,\pi)$. Then

\par
$(1^o)$ For any $\lambda \in \sigma_p(\mathcal{L}[q])$, nonlinear boundary value problem \eqref{eq:Nonl}$|_{\delta=1}$ has no non-zero weak solution with $k-1$ or less  roots in $(0,\pi)$, and  problem \eqref{eq:Nonl}$|_{\delta=-1}$ has no non-zero weak solution with $k-1$ or more  roots in $(0,\pi)$.

\par
$(2^o)$ For any $\lambda\in (\lambda_{k-1}(q_0), \lambda_k(q_0))$, $k\geq 1$, nonlinear boundary value problem \eqref{eq:Nonl}$|_{\delta=1}$  possess an infinite sequences  distinct weak solutions $(u_\delta^l)_{l=k}^\infty$. Moreover,  $u_\delta^l$, $l=k, \ldots$ has precisely $l-1$ roots in $(0,\pi)$.

\par
$(3^o)$ For any $\lambda\in ( \lambda_k(q_0), \lambda_{k+1}(q_0))$, $k\geq 1$, nonlinear boundary value problem \eqref{eq:Nonl}$|_{\delta=-1}$  possess at least $k$  distinct weak solutions $(u_\delta^l)_{l=1}^k$. Moreover, $u_\delta^l$, $l=1, \ldots, k$ has precisely $l-1$ roots in $(0,\pi)$.
\end{cor}
Here, it is assumed that $\lambda_0(q_0)=-\infty$.

\medskip

We emphasize that this result is nothing more than a generalization of the well known Sturm's nodal theorem to the nonlinear problem.
As far as we know, such a method of proving of this type of statement is new.

This paper is organised as follows.
Section 2 contains some preliminaries and the proof of Lemma \ref{lem1}. In Section 3, we give the proof of Theorems \ref{thm1}.

\section{Preliminaries}
In what follows, we denote by $\left\langle \cdot, \cdot \right\rangle $ and $\|\cdot\|_{L^2}$   the scalar product and the norm in  $L^2(0,\pi)$, respectively; $W^{1,2}(0,\pi), W^{2,2}(0,\pi)$ are usual Sobolev spaces with the norms
$$
\|u\|_{1}=\left(\int^\pi_0 |u |^2 dx+\int^\pi_0 |u'|^2 dx\right )^{1/2},~~~\|u\|_{2}=\left(\int^\pi_0 |u |^2 dx+\int^\pi_0 |u''|^2 dx\right )^{1/2}.
$$
 $W^{1,2}_0:=W^{1,2}_0(0,\pi)$ is the closure of $C^\infty_0(0,\pi)$ in the norm
$$
\|u\|_{1}=\left(\int^\pi_0 |\nabla u |^2 dx\right )^{1/2}.
$$

In what follows, we shall always suppose that  $\|\phi_k(q)\|_{L^2}=1$, $k=1,2,...$.
\begin{prop}\label{prop2}
Let $k\geq 1$ and the sequences $(q_j)_{j=1}^\infty$ in $L^2(0,\pi)$ and $(|\lambda_k(q_j)|)_{j=1}^\infty$ in $\mathbb{R}$ are bounded. Then $(\phi_k(q_j))$ is bounded in $W^{2,2}(0,\pi)$.
\end{prop}
\begin{proof}
Notice that the equation $\mathcal{L}[q]\phi_k(q_j)=\lambda_k(q_j)\phi_k(q_j)$ is equivalent to the following integral equality
\begin{equation}\label{eq:S3}
	\phi_k(q_j)(x) = \lambda_k(q_j) \int_0^1 G_0(x,\xi)\phi_k(q_j)(\xi)d\xi -\int_0^1 G_0(x,\xi) q_j(\xi) \phi_k(q_j)(\xi)d\xi,
\end{equation}
where $G_0(x,\xi)$  is the integral kernel of operator $\mathcal{L}[0]^{-1}$.
In view of that $G_0 \in C[0,1] \times C[0,1]$,  \eqref{eq:S3} implies
$$
	\|\phi_k(q_j)\|_{C[0,1]} \leq \left( \Lambda_k\max_{\xi,x}|G_0(x,\xi)|  +\max_{\xi,x}|G_0(x,\xi)|  \|q_j\|_{L^2(0,\pi)}\right)\|\phi_k(q_j)\|_{L^2(0,\pi)},
$$
where $\Lambda_k=\sup_j|\lambda_k(q_j)|$. Now taking into account that the set $\|q_j\|_{ L^2(0,\pi)}$ is bounded we derive
$$
\|\phi_k(q_j)\|_{C[0,1]}<C\|\phi_k(q_j)\|_{L^2(0,\pi)}
$$
for some $C<+\infty$ which does not depend on $j=1,2,...$.
Hence and since $\mathcal{L}[q]\phi_k(q_j)=\lambda_k(q_j)\phi_k(q_j)$ we get
\begin{align*}
\int_0^1|\phi''_k(q_j)|^2dx&\leq \int_0^1|q_j(x)\phi_k(q_j)|^2dx+\Lambda_k \int_0^1|\phi_k(q_j)|^2dx\leq \\
&	\|\phi_k(q_j)\|^2_{C[0,1]}\int_0^1|q_j(x)|^2dx+\Lambda_k \int_0^1|\phi_k(q_j)|^2dx<C_1<+\infty,
\end{align*}
where  $C_1$ does not depend on $j=1,2,...$. Thus, in view of that $\|\phi_k(q_j)\|_{L^2}=1$, we obtain
$$
\|\phi_k(q_j)\|_{W^{2,2}}<C_2<+\infty,
$$
where  $C_2$ does not depend on $j=1,2,...$.
\end{proof}

\begin{lemma}\label{lem2}
If $B$ is a bounded set  in $L^2$, then the family of operators $\mathcal{L}[q]$ is uniformly semi-bounded below on $L^2$ with respect to $q \in B$, i.e.,
$$
-\infty< \mu \leq \inf_{q\in B}\inf\{\left\langle \mathcal{L}[q] \psi, \psi\right\rangle: ~\psi\in L^2,~ \|\psi\|_{L^2}=1\}.
$$
\end{lemma}

\begin{proof} We develop an  approach proposed by Shkalikov in \cite{Shkalik}.
Write $\mathcal{L}[q] y=\mathcal{L}[0] y+Q y$, where  $\mathcal{L}[0]y=-y''(x)$, $Qy=q(x)y(x)$.  Let $a>0$ be a sufficiently large number. Introduce $R(a):=(\mathcal{L}[0]+aI)^{-1/2}$. Let us estimate the norm of the operator $R(a)QR(a)$.
For arbitrary  $f,g\in L^2(0,\pi)$ we have
$$
\left\langle R(a)QR(a)f,g\right\rangle=\left\langle QR(a)f,R(a)g\right\rangle.
$$
Denote by  $\mu_l=( l)^2 $, $\psi_l(x)=sin( l x)$, $l=1,2,...$ the eigenvalues and eigenfunctions of the operator $\mathcal{L}[0]$. Then $f=\sum_{i=1}^{\infty}f_i\psi_i(x)$, $g=\sum_{j=1}^{\infty}g_j\psi_j(x)$ in $L^2$ and
$$
\left\langle QR(a)f,R(a)g\right\rangle
=\sum_{i=1}^{\infty}\sum_{j=1}^{\infty}\frac{g_jf_i}{s_i(a)s_j(a)}
\left\langle Q\psi_i,\psi_j\right\rangle,
$$
where $s_l(a)=\sqrt{\mu_l+a}$, $l=1,2,\ldots$.
Observe
$$
\left |\left\langle Q\psi_i,\psi_j\right\rangle\right|\leq \max_{x \in [0,1]} |\psi_i(x)| \max_{x \in [0,1]} |\psi_j(x)|\int_0^{\pi}|q(s)|ds < C\|q\|_{L^2},~~i,j=1,2,\ldots,
$$
where $C<\infty$ does not depend on $i,j=1,2,...$.
Hence
\begin{eqnarray*}
	|\left\langle R(a)QR(a)f,g \right\rangle|\leq \sum_{i=1}^{\infty}\sum_{j=1}^{\infty}\frac{|g_j||f_i|}{s_i(a)s_j(a)}|\left\langle Q\psi_i,\psi_j\right\rangle |\leq&&
	\\ C \|q\|_{L_2} \sum_{i=1}^{\infty}\sum_{j=1}^{\infty}\frac{|g_j||f_i|}{s_j(a)s_i(a)}
		\leq C \|q\|_{L_2} \sqrt{\sum_{j=1}^{\infty}|g_j|^2} \sqrt{\sum_{i=1}^{\infty}|f_i|^2}\left(\sum_{i=1}^{\infty}\frac{1}{(s_i(a))^{2}}\right)=
	 &&\\C \,\rho(a) \,\|q\|_{L_2} \|f\|_{L_2}&&,
\end{eqnarray*}
where $\rho(a)=\sum_{i=1}^{\infty}\frac{1}{(s_i(a))^{2}}$.
Thus we obtain
\begin{equation}\label{est}
\left|\left\langle R(a)QR(a)f,g\right\rangle\right|\leq  C\rho(a) \|q\|_{L_2} \|f\|_{L_2},~~~~\forall f,g \in L^2.
\end{equation}
Set $h=R(a)v$ for $v \in L^2(0,\pi)$. Since $\|q\|_{L_2}$ is bounded on $B$ and $\rho(a)\rightarrow 0$ as $a\to +\infty$, we obtain that
for sufficiently large $a$ it satisfies
$$
\left\langle (\mathcal{L}[0]+aI+Q)h,h\right\rangle=  \left\langle v,v\right\rangle+ \left\langle R(a)QR(a)v,v\right\rangle>0,~~\forall v \in L^2.
$$
Thus for sufficiently large $a$ and for any $\psi\in L^2$ such that $\|\psi\|_{L^2}=1$ there holds
$$
\left\langle \mathcal{L}[q] \psi, \psi\right\rangle> -a \left\langle \psi, \psi\right\rangle=-a>-\infty, ~~~\forall q \in B.
$$
This completes the proof.

\end{proof}

%

From Lemma \ref{lem2} it follows immediately
\begin{cor}\label{cor2}
	If $B$ is a bounded set  in $L^2$, then $\lambda_1(q)\geq \mu>-\infty$ for all $q \in B$, where $\mu$ does not depend on $q \in B$.
\end{cor}

\begin{lemma}\label{lem3}
For $k\geq 1$, the map 	$\lambda_k(\cdot): L^2(0,\pi) \to \mathbb{R}$  is continuously
differentiable with the Fr\'echet-derivative
\begin{equation}\label{eq:Val}
	D\lambda_k(q)(h)=\frac{1}{\|\phi_k(q)\|^2_{L^2}}\int^\pi_0 \phi_k^2(q) h\, dx, ~~\forall \,q,h \in L^2.
\end{equation}
\end{lemma}
\begin{proof} Since $\lambda_k(q)$ is isolated,  Corollary 4.2 from \cite{ZettlM} implies that
$\lambda_k(q)$ is Fr\'echet differentiable and \eqref{eq:Val} holds. By the analyticity property (see  \cite{Poschel}, page 10), the map $\phi_k(\cdot): L^2(0,\pi) \to W^{2,2}(0,\pi)$ is analytic.
Due to the Sobolev theorem,
the embedding $W^{2,2}(0,\pi) \subset L^4(0,\pi)$ is continuous. Hence the map $\phi_k(\cdot): L^2(0,\pi) \to L^4(0,\pi)$ is continuous and therefore the norm of the derivative $D\lambda_1(q)$ continuously depends on $q \in L^2(0,\pi)$. This implies  that $\lambda_k(q)$ is continuously differentiable in $L^2(0,\pi)$.
\end{proof}

{\it Proof of Lemma \ref{lem1}}

\begin{proof}
We shall give the proof only for \textbf{(i)}. The proof of  \textbf{(ii)} is  in a similar manner.

 Let $\lambda\geq \lambda_k(q_0)$, $k\geq 1$ and $\delta=1$.	Suppose, contrary to our claim, that \eqref{eq:Nonl} has non-zero weak solution $u$ with $k-1$ or less  roots in $(0,\pi)$.	 Consider the following two equalities
\begin{align*}
	&u''+(-q_0+\lambda+u^2)u=0\\
	&\phi_k''+(-q_0+\lambda_k)\phi_k=0
\end{align*}
Observe that $-q_0+\lambda+u^2>-q_0+\lambda_k$. However this implies by the Sturm Comparison Theorem that $u$ should has more then $\phi_k$ roots in $(0,\pi)$  that is more then $k-1$ roots in $(0,\pi)$ which is  a contradiction.

\end{proof}

\section{Proof of the main result}
 We give the proof only for the case $k=2$; the other cases are left to the reader. For the proof in the case $k=1$  and $\lambda>\lambda_1(q_0)$ see also \cite{ilValMatZam, ilVal}.

Let $\lambda^* \in \mathbb{R}$. Consider the following minimization problem
\begin{equation}\label{varM2}
\hat{Q}_{\lambda^*}=\inf\{Q(q):\,\, \lambda^*=\lambda_2(q), \,\,q \in L^2(0,\pi)\},
\end{equation}
where  $Q(q):=||q_0-q||^2_{L^2}$, $q \in L^2(0,\pi)$.

Let $q_j\in L^2(0,\pi)$, $j=1,2,...$ be a minimizing sequence of this problem, i.e.,
 $\lambda_2(q_j)=\lambda^*$ and $Q(q_j) \to \hat{Q}_{\lambda^*}$.
Observe that if $\|q_j\|^2_{L^2} \to +\infty$, then $||q_0-q_j||^2_{L^2} \to +\infty$, i.e., $Q(q)$ is a coercive functional. Thus the sequence $q_j$ is bounded in $L^2(0,\pi)$, and by the Banach-Alaoglu theorem there exists a subsequence which we again denote  $(q_j)$ such that
$
q_j \rightharpoondown \hat{q}~~\mbox{as}~~j\to \infty
$
weakly in  $L^2(0,\pi)$ for some $\hat{q} \in L^2(0,\pi)$.

%

Consider sequences of eigenfunctions $(\phi_1(q_j))$ and $(\phi_2(q_j))$. By assumption, $\lambda^*=\lambda_2(q_j)$ for all $j=1,2,...$. Furthermore, in view of that $q_j$ is bounded in $L^2(0,\pi)$, Corollary \ref{cor1} entails that the sequence $\lambda_1(q_j)$ is bounded below. Hence and since $\lambda_1(q_j)<\lambda_2(q_j)=\lambda^*$, for $j=1,2,...$, we conclude that $|\lambda_1(q_j)|$ is bounded. Thus by Propsition \ref{prop2} the sequences $\phi_1(q_j)$ and $\phi_2(q_j)$ are bounded in $W^{2,2}(0,\pi)$. From this by the Sobolev embedding theorem there exist subsequences which we again denote by $\phi_1(q_j)$ and $\phi_2(q_j)$  such that
\begin{equation} \label{eq:S5}
\phi_1(q_j) \to \phi_1^*, ~~\phi_2(q_j) \to \phi_2^*~~~\mbox{as}~~~j\to +\infty
\end{equation}
strongly in $W^{1,2}(0,\pi) $ and $C^1[0,\pi]$ for some $\phi_1^*,\phi_2^* \in W^{1,2}_0(0,\pi) \cap C^1[0,\pi]$. Notice that, since $\|\phi_1(q_j))\|_{L^2}=1$, $\|\phi_2(q_j))\|_{L^2}=1$, for every $j=1,2,...$, it follows that $\phi_1^*,\phi_2^*\neq 0$

 Furthermore, we may assume, by passing to a subsequence if necessary, that   $\lambda_1(q_j) \to \lambda^*_1$ as $j\to \infty$ for some $\lambda^*_1 \in \mathbb{R}$.

Let $m=1,2$. Then
\begin{align*}\label{eq:S6}
\phi_m(q_j) = \lambda_m(q_j) \int_0^1 G_0&(x,\xi)(\phi_m(q_j)(\xi)-\phi^*_m(\xi))d\xi \\
&-\int_0^1 G_0(x,\xi) q_j(\xi) (\phi_m(q_j)(\xi)-\phi^*_m(\xi))d\xi+ \nonumber\\
+&\lambda_m(q_j) \int_0^1 G_0(x,\xi)\phi^*_m(\xi)d\xi -\int_0^1 G_0(x,\xi) q_j(\xi) \phi^*_m(\xi)d\xi.
\end{align*}
for every $j=1,2,...$. Hence strong convergences \eqref{eq:S5} and the weak convergence $q_j \rightharpoondown \hat{q}$ in $L^2(0,\pi)$ imply
\begin{equation} \label{eq:S8}
 \phi^*_m(x)=\lambda^*_m \int_0^1 G_0(x,\xi)\phi^*_m(\xi)d\xi -\int_0^1 G_0(x,\xi) \hat{q}(\xi) \phi^*_m(\xi)d\xi, ~~~m=1,2,
\end{equation}
and therefore
\begin{equation} \label{eq:S7}
-\frac{d^2}{dx^2}\phi^*_m(x)+\hat{q}(x)\phi^*_m(x)=\lambda^*_m \phi^*_m(x),\,\,\,\, x\in (0,\pi), ~~~m=1,2.
\end{equation}
This means that $(\lambda^*_1, \phi^*_1)$  and $(\lambda^*_2, \phi^*_2)$ coincide with some eigenpairs of the operator $\mathcal{L}[\hat{q}] $, i.e.,
\begin{equation}
	\lambda^*_m=\lambda_{i_m}(\hat{q}),~~~~\phi^*_m=\phi_{i_m}(\hat{q}), ~~~m=1,2,
\end{equation}
for some $i_1, i_2 \in \mathbb{N}$.
Let us show that
$ i_m=m$ for $m=1,2$.
By the Sturm comparison theorem (see e.g.,\cite{zetl}) for each $j=1,2, ... $, every eigenfunction $\phi_m(q_j)(x)$ , $m=1,2$ has precisely $m-1$ roots.

This and the strong convergences \eqref{eq:S5} in $C^1[0,1]$ yield that the limit function $\phi^*_m$ may has at most $m-1$ roots.
Hence we get that $i_2 \leq 2$ and
$i_1=1$, i.e.,
$\lambda^*_1=\lambda_1(\hat{q})$
is a principal eigenvalue of $\mathcal{L}[\hat{q}] $.

Since $\left\langle \phi_1(q_j),\phi_2(q_j)\right\rangle=0$ for all $j=1,2,...$, by passing to the limit we have $\left\langle \phi_1^*,\phi_2^*\right\rangle=0$. Thus $\phi_1^*\neq \phi_2^*$  and therefore
$$
i_2 =2, ~~~\lambda^*_2=\lambda_2(\hat{q}).
$$

Thus $\hat{q}$ is an admissible point for minimization problem \eqref{varM2}. Now taking into account that the weak  convergence  $q_j \rightharpoondown \hat{q}$ in $L^2$ imply
$$
Q(\hat{q})\leq \hat{Q}_{\lambda^*}.
$$
we obtain  that $Q(\hat{q})= \hat{Q}_{\lambda^*}$. Thus  $\hat{q}$ is a solution of \eqref{varM2}.
This concludes the proof of assertion $(1^o)$, Theorem \ref{thm1}.

Let us prove $(2^o)$. Assume $\lambda^*\neq \lambda_2(q_0)$.

Since $Q$ and $\lambda_2(\hat{q})$ are $C^1$-functionals in $L^2$, the Lagrange multiplier rule implies
\begin{equation}
	\mu_1 DQ(\hat{q})(h)+\mu_2D\lambda_2(\hat{q})(h) =0,~~ \forall h \in L^2,
\end{equation}
where $\mu_1,\mu_2 $ such that $|\mu_1|+|\mu_2|\neq 0$. Thus by \eqref{eq:Val} we deduce
\begin{equation}
	\int_\Omega (-2\mu_1  (q_0-\hat{q})+\mu_2\phi_2^2(\hat{q})) h\, dx  =0, \,\, \forall h \in L^2,
\end{equation}
where $\|\phi_2(\hat{q})\|_{L^2}=1$. Hence,
$$
2\mu_1(q_0-\hat{q})=\mu_2 \phi_2^2(\hat{q})~~\mbox{a.e. in}~~\Omega.
$$
Observe that
 $\mu_1 \neq 0,\mu_2\neq 0$. Indeed, if $\mu_1=0$, then $\phi_2(\hat{q})=0$ a.e. in $\Omega$, which is a contradiction. Suppose $\mu_2 =0$, then $q_0=\hat{q}$  a.e. in $\Omega$ and consequently $\lambda^*=\lambda_2(q_0)$ which contradicts our assumption  $\lambda^*\neq \lambda_2(q_0)$. Thus we have
\begin{equation}
	\hat{q}=q_0-\nu \phi_2^2(\hat{q})~~\mbox{a.e. in}~~\Omega,
\end{equation}
with some constant $\nu \neq 0$. Substituting this into the equality
$$
 -\phi_2''(\hat{q})+\hat{q} \phi_2(\hat{q})=\lambda^* \phi_2(\hat{q})
$$
we obtain
\begin{equation}\label{eq:phi}
- \phi_2''(\hat{q})+q_0 \phi_2(\hat{q})=\lambda^* \phi_2(\hat{q})+\nu\phi_2^3(\hat{q}).
\end{equation}
Thus, $\hat{u}=|\nu|^\frac{1}{2}\phi_2(\hat{q})$ satisfies \eqref{eq:Nonl}  and  $\hat{q}=q_0-\delta\hat{u}^{2}$ a.e. in $\Omega$ with $\delta$=sign($\nu$).
Now taking into account Lemma \ref{lem1},  we infer that it must be $\delta=1$ if  $\lambda< \lambda_k(q_0)$, and $\delta=-1$ if $\lambda> \lambda_k(q_0)$.
This concludes the proof of $(2^o)$.

The proof of $(3^o)$ follows immediately since  $\hat{u}(x)=\phi_k(\hat{q})(x)\cdot\|\hat{u}\|_{L^2}$ and by Sturm's nodal theorem the eigenfunction $\phi_k(\hat{q})(x)$ of $\mathcal{L}[\hat{q}]$  has precisely $k-1$ roots.

\medskip

\textit{Proof of Corollary \ref{cor1}. }
For $\lambda= \lambda_k(q_0)$, $k=1,\ldots$, by \textbf{(i)} Lemma \ref{lem1} it follows that   nonlinear boundary value problem \eqref{eq:Nonl}$|_{\delta=1}$ has no non-zero weak solution with $k-1$ or less  roots in $(0,\pi)$ whereas	by \textbf{(ii)} Lemma \ref{lem1},  problem \eqref{eq:Nonl}$|_{\delta=-1}$ has no non-zero weak solution with $k-1$ or more  roots in $(0,\pi)$ which implies  ($1^o$).

Since  $\lambda_1(q)<\lambda_2(q)< \ldots$, the proofs of assertions ($2^o$), ($3^o$) immediately follow from ($2^o$), ($3^o$) Theorem \ref{thm1}.


\begin{thebibliography}{10}


%
\bibitem{ambar} V. Ambarzumian, \"Uber eine frage der eigenwerttheorie. Zeitschrift f\"ur Physik A Hadrons and Nuclei 53 (9) (1929), 690-695.



\bibitem{borg} G. Borg,   Eine umkehrung der Sturm-Liouvilleschen eigenwertaufgabe. Acta Mathematica, 78 (1) (1946), 1-96.



	\bibitem{edmund} D. E. Edmunds,  W. D. Evans, Spectral theory and differential operators. Vol. 15. Oxford: Clarendon Press, 1987.

\bibitem{gelL}  I. M. Gel'fand, B. M. Levitan,  On the determination of a differential equation from its spectral function. Izvestiya Rossiiskoi Akademii Nauk. Seriya Matematicheskaya, 15 (4) (1951), 309-360.

\bibitem{ilValMatZam} Y. Sh. Ilyasov,   N. F. Valeev,
On an inverse optimization spectral problem and a related nonlinear boundary value problem. Math. Zametki, 104 (4) (2018), 621-625.

\bibitem{ilVal} Y. Sh. Ilyasov,   N. F. Valeev, On nonlinear boundary value problem corresponding to $ N $-dimensional inverse spectral problem. arXiv preprint arXiv:1803.01495 (2018), 1-11.







\bibitem{ZettlM} M. M\"oller ,  A. Zettl, Differentiable dependence of eigenvalues of operators in Banach spaces, Journal of Operator Theory.  (1996), 335-355.

\bibitem{Poschel}  J. P\"oschel, and E. Trubowitz,  Inverse spectral theory, volume 130 of Pure and Applied Mathematics, 1987.






\bibitem{Shkalik} A. A. Shkalikov,  Perturbations of self-adjoint and normal operators with discrete spectrum. Russian Mathematical Surveys.   V. 71(5) (2016) , 907 2016.

\bibitem{GilTrud} D. Gilbarg, N. S. Trudinger,  Elliptic partial differential equations of second order. Springer, 2015.







\bibitem{zetl} A. Zettl,  Sturm-Liouville theory. No. 121. American Mathematical Soc., 2005.


\end{thebibliography}
\end{document}